\documentclass{birkjour}
 \newtheorem{thm}{Theorem}[section]
 \newtheorem{cor}[thm]{Corollary}
 \newtheorem{lem}[thm]{Lemma}
 
 \theoremstyle{definition}
 
 \theoremstyle{remark}

 \numberwithin{equation}{section}

\begin{document}

%
%
%
%
%
%
%
%
%

\title [Biharmonic Hypersurfaces]
{Biharmonic Hypersurfaces With Recurrent Operators In The Euclidean Space}

\author[N. Mosadegh]{N. Mosadegh}

\address{Department of Mathematics,
 Azarbaijan Shahid Madani University,\\
Tabriz 53751 71379, Iran}

\email{n.mosadegh@azaruniv.ac.ir}

\author{E. Abedi}
\address{Department of Mathematics,
 Azarbaijan Shahid Madani University,\\
Tabriz 53751 71379, Iran}
\email{esabedi@azaruniv.ac.ir}
\subjclass{Primary 53C42; Secondary 53C43, 53B25}

\keywords{Biharmonic hypersurfaces, Recurrent operators}

\date{January 1, 2004}
\dedicatory{}

\begin{abstract}
We show how some of well-known recurrent operators such as recurrent curvature operator, recurrent Ricci operator, recurrent Jacobi operator, recurrent shape and Weyl operators have the significant role for biharmonic hypersurfaces to be minimal in the Euclidean space.
\end{abstract}
\footnotetext[1]{The first author is as corresponding author.}
\maketitle
\section{Introduction}

The phrase harmonic map $f:(M,g)\rightarrow (N, h)$ between two the Riemannian manifolds is which that refers to the critical points of the energy functional $E(f)=\frac{1}{2}\int_M |df|^2\star 1$. The studying $K$-harmonic maps, correspondingly, $k$-harmonic submanifolds began with J. Eells and L. Lemair. It was proposed to investigate $K$-harmonic maps as critical points of the functional 
\begin{eqnarray}
E: C^{\infty}(M, N)\longrightarrow R\ \ \ \,
E_K(f)=\int_M \|d+ d^{\star} \|^2 f\star 1
\end{eqnarray}
where $d$ and $d^{\star}$ are the exterior differentation and codifferentation on the vector bundle on $M$, respectively (see \cite{J,J_2}). The idea was supported in case $K=2$, which is called biharmonic maps and deal with $E_2(f)=\frac{1}{2}\int_M |\tau(f)|^2 d\upsilon$, where $\tau(f)= \textsf{trace}\nabla df$ is the tension field of $f$ \cite{G,G_2}. Furthermore, the Euler-Lagrange equation associated to $E_2$ is given by vanishing of the bitension field written as:
\begin{eqnarray}
\tau_2(f)&=&-\Delta \tau(f)- \textsf{trace} R^N(df, \tau(f))df=0.\nonumber
\end{eqnarray}
The interesting is in the non harmonic biharmonic maps which are called proper biharmonic. The first ambient spaces to investigate the proper biharmonic submanifolds are spaces of the constant sectional curvature. In this case, the biharmonic concept of submanifold in the Euclidean space with the harmonic mean curvature vector was established by B. Y. Chen. Indeed the well known conjecture was posted: any biharmonic submanifold in Euclidean space is harmonic see \cite{B}. By following the Chen's conjecture, hypersurfaces are the first class of submanifolds to be studied such that up to now, the following classification results reached .
\begin{itemize}
                                            \item Biharmonic hypersurfaces in $E^n, n = 3, 4, 5$, are minimal \cite{Chen2.13,Ha,sham};
                                            \item Biharmonic hypersurfaces in $4$-dimensional space form $H^4$ are minimal \cite{Bol};
                                             \item  The biharmonic submanifold with the constant mean curvature and biharmonic hypersurfaces with at most two distinct principal curvatures in the Euclidean space are minimal \cite{Dim};
                                              \item Biharmonic hypersurfaces with three distinct principal curvatures in $R^n$ and $S^n$ are minimal \cite{YU1,YU3};
                                          \end{itemize}
Furthermore, a result of K. Akutagawa and Maeta \cite{AK} states that the biharmonic complete submanifolds in the Euclidean space are minimal too. Motivated by the results, authors in \cite{Abedi2,Ab4} deal with the biharmonic Hopf hypersurfaces in the complex Euclidean spaces and in the odd dimensional spheres and showed they are minimal. Specifically, they proved the nonexistence result of the proper biharmonic Ricci Soliton hypersurfaces in the Euclidean space $E^{n+1}$, if the potential vector field is a principal direction.

In this survay we shall focus on the biharmonic hypersurfaces in the Euclidean space $E^{n+1}$ with an important object attaches to them is the recurrent operator. The key observation throughout is that the recurrent operators can be a property of the biharmonic hypersurfaces, which can not be proper one. Indeed, we show that the biharmonic hypersurfaces with some recurrent operators in the Euclidean space are minimal .Clearly, the results are given in sections 3, following the works in \cite{Dim,YU1}.  

 \section{preliminaries}
 
 Let $x: M^n\longrightarrow E^{n+1}$ be an isometric immersion of an $n$-dimensional hypersurface $(M^n,g)$ into the Euclidean space $E^{n+1}$. Let $\nabla$ and $\overline{\nabla}$ be the Levi-Civita connections on $M^n$ and $E^{n+1}$, respectively. Let $N$ be a local unit normal vector field to $M^n$ in $E^{n+1}$ and $\overrightarrow{H}=HN$ be the mean curvature vector field. One of the considerable equation in differential geometry is $\triangle x=-n\overrightarrow{H}$, where $\triangle$ the Laplacian-Beltrami operator is defined $\triangle$ = - \textsf{trace} $\nabla ^2$. The expressions assumed by the tension and bitension fields satisfies
\begin{eqnarray*}
\tau(x) = n \overrightarrow{H}, \ \ \ \tau_2(x) = -n \Delta\overrightarrow{ H},
\end{eqnarray*}
 then, the immersion $x$ is called biharmonic if and only if $\triangle \overrightarrow{H}=0$, where written as:
\begin{eqnarray*}
0= \triangle \overrightarrow{H}= 2A (\textsf{grad} H) +n H \textsf{grad} H+(\triangle^{\perp} H + H \textsf{trace} A^2),
\end{eqnarray*}
by identifying the bitension field in its normal and
tangent components, the main tool is obtained in the study of the proper biharmonic hypersurfaces in the Euclidean spaces.
\begin{thm}\cite{Bc}\label{00}
Let $x: M^n\longrightarrow E^{n+1}$ be an isometric immersion of an $n$-dimensional hypersurface $(M^n,g)$ into the Euclidean space $E^{n+1}$. Then $M^n$ is a biharmonic hypersurface if and only if
 \begin{eqnarray}
 \left\{
   \begin{array}{ll}
      \hbox{${\triangle}^\perp H +H \textsf{trace} A^2=0$;} \\
       \hbox{$2A(\textsf{grad} H)+nH \textsf{grad} H=0$,} \label{1.0}
   \end{array}
 \right.
 \end{eqnarray}
where $A$ denotes the Weingarten operator and $\Delta^{\perp}$ the Laplacian in the normal bundle of $M^n$ in $E^{n+1}$.
 \end{thm}
 
In the rest of the content, we deal with an orthonorma frame field $\{e_i\}_{i=1}^n$ on biharmonic hypersurface $M^n$ in such away that $e_i$ are the principal directions and $e_1=\frac{\mathsf{grad}H}{|\mathsf{grad}H|}$ and we call it is an \textit{appropriate frame field}.
 
\begin{lem}\label{2.1.3}
Let $M^n$ be a biharminic hypersurface in the Euclidean space $E^{n+1}$. Suppose that the mean curvature of $M^n$ is not constant. Then for the appropriate frame field $\{e_i\}_{i=1}^n$
\begin{eqnarray}
\nabla_{e_1}e_i=\sum_{k=1}^n \omega_{1i}^k e_k=0 \ \ \ \mathsf{for} \ \ i=1, ..., n,\ \ \nabla_{e_i}e_1=\omega_{ii}^1e_i\ \ \ \mathsf{for}\ \  i\neq 1,
\end{eqnarray}
where $\omega_{ij}^k$ are called connection forms for any $i,j, k=1,...,n$.
\end{lem}
\begin{proof}
Let $x:M^n\longrightarrow E^{n+1}$ be an isometric immersion of the biharmonic hypersurface $M^n$ with the non constant mean curvatuer. So,  there exists a point $p \in M^n$, where \textsf{grad}$H\neq 0$ at $p$ then there is an open subset $U$ of $M^n$ such that \textsf{grad}$H\neq 0$ on $U$. By Theorem \ref{00} we have \textsf{grad}$H$ is a principal direction corresponding to the unique principal curvature $\frac{-n}{2}H$. Suppose that the weingarten operator $A$ takes the form $Ae_i=\lambda_i e_i$, i.e. $e_i$ is an eigenvector of $A$ with eigenvalue $\lambda_i$. We choose $e_1$ such that $e_1$ is parallel to \textsf{grad}$H$ where it expresses \textsf{grad}$H=\sum_{i=1}^n(e_iH)e_i$, this shows that $(e_1H)\neq 0$ and $(e_iH)=0$ for any $i=1, ..., n$. For following our approach, we need to estimate the connection forms $\omega_{ij}^k$ which is given $\nabla_{e_i}e_j=\sum_{i=1}^n\omega_{ij}^ke_k$. By this we have
\begin{eqnarray}\label{2.1.2}
\omega_{ki}^i=0,\ \ \ \omega_{ki}^j+ \omega_{kj}^i=0 \ \ \ i\neq j, \ \ i, j, k=1, ..., n
\end{eqnarray}
since $\nabla_{e_k}<e_i, e_j>=0$. Morever, by the above and the Codazzi equation we find 
\begin{eqnarray}
e_k(\lambda_j)e_i + (\lambda_i-\lambda_j)\omega_{ki}^je_j=e_i(\lambda_k)e_k+ (\lambda_k-\lambda_j)\omega_{ik}^je_j
\end{eqnarray}
which yields
\begin{eqnarray}\label{2.3.7}
e_i(\lambda_j)=(\lambda_i-\lambda_j)\omega_{ji}^j\nonumber\\
(\lambda_i-\lambda_j)\omega_{ki}^j=(\lambda_k-\lambda_j)\omega_{ik}^j
\end{eqnarray}
for distinct $i, j$ and $k$ where $i,j, k=1, ..., n$.
Now, we set $\lambda_1=\frac{-n}{2}H$, this implies $(e_1\lambda_1)\neq 0$ and $(e_i \lambda_1)=0$ for any $i=2, ..., n$. Then we have
\begin{eqnarray}
0=[e_i, e_j]\lambda_1=(\omega_{ij}^1-\omega_{ji}^1)(e_1\lambda_1) , \ \ \  2\leq i,j \leq n, \ \ \ i\neq j
\end{eqnarray}
which shows 
\begin{eqnarray}\label{1.2.1}
\omega_{ij}^1=\omega_{ij}^1, \ \ \  2\leq i,j \leq n, \ \ \ i\neq j.
\end{eqnarray}
Observe that for indices $j=1$ and $ 2\leq i, k\leq n$ the equation (\ref{2.3.7}) follows
\begin{eqnarray*}
(\lambda_i-\lambda_1)\omega_{ki}^1=(\lambda_k-\lambda_1)\omega_{ik}^1,
\end{eqnarray*}
because of uniqueness of $\lambda_1$ and  (\ref{1.2.1}) by the above we have
\begin{eqnarray*}
\omega_{ij}^1=\omega_{ji}^1=0, \ \ \ i\neq j,\ \ \  2 \leq i, j \leq n.
\end{eqnarray*}
On the one hand, from (\ref{2.1.2}) it follows $\omega_{k1}^1=0$ and $\omega_{k1}^j+\omega_{kj}^1=0$ for any $i,j, k=1, ..., n$. Then, $\omega_{1i}^1=\omega_{11}^i=0$ where $i=1, ..., n$. So, 
\begin{eqnarray*}
\omega_{ij}^1=\omega_{ji}^1=0, \ \ \ i\neq j,\ \ \  i, j=1, ...,  n.
\end{eqnarray*}
Afterall, putting this all together, give the claime.
\end{proof}

 \section{Biharmonic hypersurfaces in the Euclidean space $E^{n+1}$ with the recurrent operators}

Let $T$ be a tensor on the Rimannian manifold $M^n$, then $T$ is said to be\textit{recurrent} if there exists a certain $1$-form $\eta$ on $M^n$ such that for any $X$ tangent to $M^n$ satisfies $\nabla_X T= \eta(X) T$. So, the recurrent $(1,1)$-tensors are extension of the parallel one.
\begin{thm}
Let $M^n$ be a biharmonic hypersurface with the recurrent Ricci operator in the Euclidean space $E^{n+1}$, then $M^n$ is a minimal hypersurface.
\end{thm}
\begin{proof}
Let $x: M^n \longrightarrow E^{n+1}$ be an isometric biharmonic immersion. Consider the appropriate frame field $\{e_i\}_{i=1}^n$ om $M^n$ then, by the Guass equation we have
$\textsf{Ric}(e_j)=\alpha_j e_j$ for any $j=1, ..., n$ where $\alpha_j=nH\lambda_j- \lambda_j^2$. Since the Ricci operator is recurrent i.e. $\big(\nabla_{X}Ric\big)Y=\eta(X)Ric(Y)$ for $X$ and $Y$ tangent to $M^n$, we get
\begin{eqnarray}\label{2.3.4}
\nabla_{e_i}Ric(e_j)= \eta(e_i)\alpha_j e_j+ \sum_{k}\omega_{ij}^k\alpha_ke_k, \ \ \ i,j=1, ..., n
\end{eqnarray}
Now, by derivative from both sides of $Ric(e_j)=\alpha_j e_j$ we have
\begin{eqnarray}\label{2.3.5}
\nabla_{e_i}Ric(e_j)=(e_i \alpha_j)e_j+ \alpha_j\sum_{k}\omega_{ij}^ke_k, \ \ \ i,j=1, ..., n
\end{eqnarray}
By following our approach we show that, at most three of principal curvatures are distinct at each point of $M^n$ and the result follows by \cite{YU3}. By the Lemma \ref{2.1.3} we have $\omega_{ij}^1=0=\omega_{ij}^j$ for any $i\neq j$. So, the conditions (\ref{2.3.4}) and (\ref{2.3.5}) imply 
\begin{eqnarray}\label{2.3.6}
\eta(e_i)\alpha_j=e_i\alpha_j, \ \ \ i\neq j, \ \ \ i,j=1, ..., n.
\end{eqnarray}
So, by the above we get two cases 
\begin{itemize}
\item for some $1\leq j \leq n$, $\alpha_j= 0$.
\item $\alpha_j\neq 0$ for all $j=1, ..., n$.
\end{itemize}
In the first case, from (\ref{2.3.4}), (\ref{2.3.5}) and (\ref{2.3.6}) we have $\sum_{k}\omega_{ij}^k\alpha_k e_k=0$. So, $\alpha_k=0$ for any $k\neq 1$. Since $\alpha_k=0$ is a second order equation of $\lambda_k$ then it has at most to distinct roots such as $\lambda_k=0$ and $\lambda_k=nH$. Therefor, $M^n$ has at most three distinct principal curvatures $0$, $nH$ and $\lambda_1=\frac{n}{2}H$.
In the second case, by taking (\ref{2.3.6}), it follows $\eta(e_i)=e_i \textsf{ln} \alpha_j$ since $\alpha_j \neq 0$ for any $ 1 \leq j \leq n$ and $i\neq j$ which implies $\frac{\textsf{ln}\alpha_j}{\alpha_1}=$constant such that gives $\alpha_j=c \alpha_1$ for a positive constant $c$ for any $j\neq 1$. By the above we have at most two distinct roots for $\lambda_j(j\neq 1)$. Therefore, we have at most three distinct principal curvature by adding $\lambda_1=\frac{-nH}{2}$. Then by following Yu Fu studying we get the result.
\end{proof}

\begin{thm}
Let $ M^n$ be a biharmonic hypersurface in the Euclidean space $E^{n+1}$ with the recurrent curvature operator. Then $M^n$ is minimal.
\end{thm}
\begin{proof}
Let $x: M^n \longrightarrow E^{n+1}$ be an isometric immersion of a biharmonic hypersurface $M^n$ in the Euclidean space $E^{n+1}$. Chossing the appropriate frame field $\{e_1,...e_n\}$, the Guass equation yields $R(e_i,e_j)e_k=0$, for distinct $i, j$ and $k$. According to the assumption the curvature operator $R$ is recurrent, i.e., $(\nabla_{X}R(Y,  Z)) W= \eta(X)R(Y, Z)W$ for all $X, Y, Z$ and $W$ tangent to $M^n$ so, $(\nabla_{e_i}R(e_j, e_k)) e_l= \eta(e_i)R(e_j, e_k)e_l=0$.
 Then, take the Guass equation we have
 \begin{eqnarray*}
0= (\nabla_{e_i}R(e_j,e_k))e_l&=&- R(e_j,e_k)\nabla_{e_i}e_l\\
 &=& \omega_{il}^k\lambda_k\lambda_j e_j-\omega_{il}^j\lambda_j\lambda_k e_k,
 \end{eqnarray*}
 where $i, j, k, l\neq 1$ beacause by Lemma \ref{2.1.3} $\omega_{ij}^1=0$ for $i\neq j$ . Then, from the linear independence of $\{e_i\}$ follows that  
$\omega_{il}^k\lambda_k\lambda_j =0 $. 
 Now, for all nonzero principal curvatures it follows $\omega_{il}^k=0$ for distinct indices. Thus, all we need is to use the Codazzi equation (\ref{2.3.7})  in which 
\begin{eqnarray*}
 0=(\lambda_l-\lambda_k)\omega_{il}^k=(\lambda_i-\lambda_k)\omega_{li}^k,
\end{eqnarray*}
this yields $\lambda_i=\lambda_k$ or $\omega_{li}^k=0$ for $j\neq k$. In particular if $\omega_{li}^k=0$, then the Codazzi equation implies $\lambda_l=\lambda_k$ too. 
Hence, by the above there exist at most two distinct principal curvatures at each point of $M^n$. Note that $\lambda_1=\frac{-n}{2}H$ that is corresponding to the principal direction $e_1=\frac{\mathsf{grad}H}{|\mathsf{grad}H|}$. Now, by following the studying in \cite{Dim}, we obtain the result.
\end{proof}

Now, directly by the above theorem we will have the following result
\begin{cor}
 The biharmonic locally symmetric hypersurfaces in the Euclidean space $E^{n+1}$ are minimal.
\end{cor}

\begin{thm}
Let $M^n$ be a biharmonic hypersurface in the Euclidean space $E^{n+1}$, with the recurrent Jacobi operator $R_X$ for any $X\in \Gamma(T(M^n))$ . Then $M^n$ is minimal.
\end{thm}
\begin{proof}
Let $x: M^n \longrightarrow E^{n+1}$ be an isometric immersion of a biharmonic hypersurface $M^n$ in the Euclidean space $E^{n+1}$. Now we use the assumption that the Jacobi operator is recurrent, i.e., $(\nabla_Y R_X)  (Z)= \eta(Y)R_X (Z)$ for all $X, Y$ and $Z$ tangent to $M^n$. Consider the appropriate frame field $\{e_i\}_{i=1}^n$ and the Guass equation then we see that the recurrent Jacobi operator expresses 
\begin{eqnarray*}
\nabla_{e_i}R_{e_j} (e_k)&=& \eta(e_i)R_{e_j} (e_k) + R_{e_j}(\nabla_{e_i}e_k)\nonumber\\
&=&\eta(e_i)R(e_k, e_j)e_j + R(\nabla_{e_i}e_k, e_j)e_j\nonumber\\
&=&-\eta(e_i)\lambda_j \lambda_k e_k-\lambda_j\sum_{l=1, l\neq j}^n \omega_{ik}^l\lambda_l e_l.
\end{eqnarray*}
Note that
\begin{eqnarray*}
\nabla_{e_i}R_{e_j} (e_k)&=& \nabla_{e_i}R(e_k, e_j)e_j\nonumber\\
&=&-e_i(\lambda_j \lambda_k)e_k- \lambda_j\lambda_k \sum_{l=1}^n \omega_{ik}^le_l,
\end{eqnarray*}
 comparing the components follows that $ \lambda_j\sum_{l=1, l\neq j}^n \omega_{il}^l\lambda_l e_l= \lambda_j\lambda_k\sum_{l=1}^n\omega_{ik}^le_l$. If $\lambda_j\neq 0$, then
\begin{eqnarray*}
\sum_{l=2}^n(\lambda_l-\lambda_k)\omega_{ik}^le_l-\lambda_j\omega_{ik}^je_j=0.
\end{eqnarray*} 
One consequence of the above is that $\lambda_l=\lambda_k$ for $2 \leq l, k \leq n$. Then take $\lambda_1=-\frac{n}{2}H$ and its uniquness turns out that there are two distinct principal curvatures at each points of $M^n$. Furthermore, because $\lambda_j\neq 0$ so $\omega_{ik}^j=0$ that the Codazzi equation follows $(\lambda_i- \lambda_j)\omega_{ki}^j=0$, which yields $\lambda_i= \lambda_j$ for $i\neq j$. Similarly, we get the same result. Now, by following the work in \cite{Dim} we obtain what was claimed.
\end{proof}

\begin{thm}
Let $M^n$ be a biharmonic hypersurface with the recurrent Weyl operator $W_{X,Y}$ for any $X, Y \in \Gamma(T(M^n))$ in the Euclidean space $E^{n+1}$. Then $M^n$ is  minimal.
\end{thm}
\begin{proof}
let $x: M^n \rightarrow E^{n+1}$ be an isometric immersion of a biharmonic hypersurface $M^n$ in the Euclidean space $E^{n+1}$.  In this case we see that with the appropriate frame field $\{e_i\}_{i=1}^n$ on $M^n$, the Weyl operator $W_{e_i, e_j}(e_k)$ vanishes for distinct indices, since 
\begin{eqnarray*}
W_{e_i,e_j}(e_k) &=& R(e_i,e_j)e_k- \frac{1}{n-2}\{\textsf{Ricci}(e_j,e_k)e_i- \textsf{Ricci}(e_i,e_k)e_j\\
 &+& g(e_j,e_k)\textsf{Ricci}(e_i)-g(e_i,e_k)\textsf{Ricci}(e_j)\}\\
 &+& \frac{s}{(n-1)(n-2)}\{g(e_j,e_k)e_i- g(e_i,e_k)e_j\},
\end{eqnarray*}
 where $R$ and $s$ are the curvature tensor and the scalar curvature, respectively and all terms are zero. Note that, 
 $W_{e_i, e_j}(e_j)= \alpha e_i$ where $\alpha=\lambda_i\lambda_j-(\lambda_i+ \lambda_j)(H-\lambda_i-\lambda_j)+\frac{s}{n-2}$. Consider the assumption that the Weyl operator is recurrent, i.e., $(\nabla_V W_{X,Y})(Z)= \eta(V)W_{X,Y}(Z)$, for all $X, Y, Z$ and $V$ tangent to $M^n$. In particular, it shows
\begin{eqnarray*}
0= \nabla_{e_j}W_{e_i, e_j}(e_1)=W_{e_i, e_j}(\nabla_{e_j}e_1)&=&\omega_{j1}^jW_{e_i,e_j}(e_j),\\
&=&\omega_{j1}^j\alpha e_i,
\end{eqnarray*} 
where by the Lemma \ref{2.1.3} $\nabla_{e_j}e_1= \omega_{jj}^1e_j$ for $j\neq 1$. Thus, $\alpha=0$ ,i.e., 
\begin{eqnarray}\label{2.3.10}
\lambda_i\lambda_j-(\lambda_i+ \lambda_j)(H-\lambda_i-\lambda_j)=a, \ \ \ i\neq j
\end{eqnarray}
in which $a=\frac{s}{2-n}$. Now, to reach the purppose we need to consider 
\begin{eqnarray}\label{2.3.11}
\lambda_i\lambda_k-(\lambda_i+ \lambda_k)(H-\lambda_i-\lambda_k)=a, \ \ \ i\neq k
\end{eqnarray}
then from (\ref{2.3.10}) and (\ref{2.3.11}) it follows 
\begin{eqnarray*}
3\lambda_i+\lambda_j+\lambda_k-H =0, \ \ \  2\leq i, j, k\leq n,
\end{eqnarray*}
which leads to that all the principal curvatures all equal. Now, take the unique principal curvature $\lambda_1=-\frac{nH}{2}$ corresponding to the principal direction $e_1=\frac{\mathsf{grad}H}{|\mathsf{grad}H|}$. So, there exist two distinct principal curvatures at each point of $M^n$. Then, by following the studying in\cite{Dim} we get the result.  
\end{proof}

\begin{thm}
Let $M^n$ be a biharmonic hypersurface with the recurrent shape operator in the Euclidean space $E^{n+1}$. Then $M^n$ is minimal.
\end{thm}
Let $x: M^n \longrightarrow E^{n+1}$ be an isometric immersion of a biharmonic hypersurface $M^n$ in the Euclidean space $E^{n+1}$. We use the assumption that the shape operator is recurrent, i.e., $(\nabla_X A)Y=\eta(X)A(Y)$ for $X$ and $Y$ tangent to $M^n$ such that for the appropriate frame field $\{e_i\}_{i=1}^n$ it satisfies
\begin{eqnarray*}
g((\nabla_{e_i} A)e_j, e_k) &=& \eta(e_i)g(\lambda_je_j, e_k)=0.
\end{eqnarray*}
then the Codazzi equation yields 
\begin{eqnarray*}
0=g((\nabla_{e_i} A)e_j,e_k)&=&(\lambda_j-\lambda_k)g(\nabla_{e_i}e_j,e_k) \\
&=& (\lambda_j-\lambda_k)\omega_{ij}^k,
\end{eqnarray*}
for $2 \leq i,j, k\leq n$ where by the Lemma \ref{2.1.3}  $\nabla_{e_i}e_j= \sum_{l=2,l\neq j}^n \omega_{ij}^le_l$. By the above, one consequence is $\lambda_j=\lambda_k$ for $ 2\leq j, k \leq n$. Add the unique principal curvatuer $\frac{-nH}{2}$ corresponding with the principal direction $e_1=\frac{\mathsf{grad}H}{|\mathsf{grad}H|}$ then it determines that there exist two distinct principal curvatures at each point of $M^n$. If $\lambda_j\neq \lambda_k$ then $\omega_{ij}^k=0$ and in this case the Codazzi equation expresses
\begin{eqnarray}
0=(\lambda_j-\lambda_k)\omega_{ij}^k=(\lambda_i-\lambda_k)\omega_{ji}^k, 
\end{eqnarray}
so, $\lambda_i=\lambda_k$ where $ 2\leq i, k \leq n$. Similarly, take the $\lambda_1=\frac{-nH}{2}$ it leads to there are two distinct principal curvatures at each point. Then we get the result by the work in \cite{Dim}.

 

\end{document}